\pdfoutput=1

\documentclass[reqno]{amsart}

\usepackage{lmodern}
\usepackage[T1]{fontenc}
\usepackage[utf8]{inputenc}
\usepackage[english]{babel}
\usepackage{microtype} %Better font spacing

\usepackage{amsmath,amssymb,amsthm,mathrsfs,latexsym,mathtools,mathdots,enumerate,tikz}
\usepackage[enableskew,vcentermath]{youngtab}
\usepackage{ytableau}

\usetikzlibrary{positioning}
\usepackage[enableskew,vcentermath]{youngtab}
\usepackage[capitalize]{cleveref}
\usepackage{bm}
\usepackage{todonotes}

\usepackage{graphicx}

\usepackage{ifpdf}
\graphicspath{{./figures/}}
\usepackage{epstopdf}
\DeclareGraphicsExtensions{.png,.jpg,.eps,.epsf}

\newtheorem{theorem}{Theorem}
\newtheorem{proposition}[theorem]{Proposition}
\newtheorem{lemma}[theorem]{Lemma}
\newtheorem{corollary}[theorem]{Corollary}

\theoremstyle{definition}
\newtheorem{definition}[theorem]{Definition}
\newtheorem{example}[theorem]{Example}
\newtheorem{remark}[theorem]{Remark}
\newtheorem{question}[theorem]{Question}

\newcommand{\defin}[1]{\emph{#1}}

\newcommand{\setN}{\mathbb{N}}
\newcommand{\setZ}{\mathbb{Z}}

\newcommand{\setR}{\mathbb{R}}

\newcommand{\avec}{\mathbf{a}}

\newcommand{\xvec}{\mathbf{x}}
\newcommand{\yvec}{\mathbf{y}}
\newcommand{\zvec}{\mathbf{z}}

\newcommand{\tvec}{\mathbf{t}}

\newcommand{\polyP}{\mathcal{P}}

\newcommand{\ssyt}{\textsc{ssyt}}
\newcommand{\svt}{\textsc{svt}}
\newcommand{\spyt}{\textsc{spyt}}
\newcommand{\rpp}{\textsc{rpp}}

\newcommand{\keytab}{\textsc{ktab}}
\newcommand{\ssaf}{\textsc{ssaf}}
\newcommand{\nawf}{\textsc{nawf}}
\newcommand{\key}{\mathcal{K}}
\newcommand{\atom}{\mathcal{A}}
\newcommand{\schur}{\mathrm{s}}
\newcommand{\schurSp}{\mathrm{sp}}
\newcommand{\grothG}{\mathrm{G}}
\newcommand{\grothg}{\mathrm{g}}
\newcommand{\schubert}{\mathfrak{S}}

\newcommand{\symS}{S}

\newcommand{\stat}{\sigma}
\newcommand{\statT}{\tau}

\newcommand{\tabFamT}{\mathcal{T}}

\DeclareMathOperator{\dn}{dn}
\DeclareMathOperator{\coinv}{coinv}

\title{Polynomials defined by tableaux and linear recurrences}
\author{Per Alexandersson}



\begin{document}

\begin{abstract}

We show that several families of polynomials defined via fillings of diagrams
satisfy linear recurrences under a natural operation on the shape of the diagram.
We focus on \emph{key polynomials}, (also known as Demazure characters),
and \emph{Demazure atoms}. The same technique can be applied to Hall--Littlewood polynomials and
dual Grothendieck polynomials.

The motivation behind this is that such recurrences are strongly connected with other nice properties,
such as interpretations in terms of lattice points in polytopes and divided difference operators.
\end{abstract}

\maketitle

\setcounter{tocdepth}{1}
\tableofcontents

\section{Introduction}

Using a similar technique as in \cite{Alexandersson20141}, we provide a framework for showing that
under certain conditions, polynomials encoding statistics on certain tableaux, or fillings of diagrams,
satisfy a linear recurrence. We prove that several of the classical polynomials from representation theory
fall into this category, such as (skew) Schur polynomials and Hall--Littlewood polynomials.

The main concern in this paper are the so called \emph{key polynomials}, indexed by integer partitions,
and \emph{atoms}. The key polynomials are natural, non-symmetric generalizations of Schur polynomials
and are specializations of the non-symmetric integer form Macdonald polynomials, see \cite{Mason2009} for details.

Let $\lambda$ be a fixed diagram shape, (a partition shape, skew shape, \emph{etc.})
and let $P_{k \lambda}(\xvec)$, $k=1,2,\dots,$ be a sequence of polynomials
which are generating functions of fillings of shape $k\lambda$.
For partitions, $k\lambda$ is simply elementwise multiplication by $k$.
There are several reasons why one would be interested in showing that a such sequence satisfies a linear recurrence:

\begin{enumerate}
 \item To obtain hints about the existence or non-existence of formulas of certain type.
For example, the Weyl determinant formula for Schur polynomials implies that the ordinary Schur polynomials satisfy a linear recurrence.

\item To obtain evidence for alternative combinatorial interpretations of the tableaux involved.
For example, the skew Schur polynomials can be obtained as lattice points in certain marked order polytopes,
called Gelfand--Tsetlin polytopes. Such a polytope interpretation implies the existence of
a linear recurrence relation. 

\item To prove polynomiality in $k$ of the number of fillings of shape $k\lambda$.

\item To obtain results about asymptotics. For example, in  \cite{Alexandersson12schur} we used such recurrences
to give a new proof of a classical result on asymptotics of eigenvalues of Toeplitz matrices.
\end{enumerate}

In the last section, we provide several examples of polynomials that satisfy such linear recurrences.
We also sketch two additional proofs in the case of key polynomials, to illustrate that
several nice properties imply the existence of a linear recurrence relation.
These methods are based on a lattice-point representation and an operator characterization of the key polynomials.
There is no straightforward way to check if a family of polynomials have such characterizations, but it is easy to
generate computer evidence that a sequence of polynomials satisfy a linear recurrence. 
Thus, proving the existence or non-existence of linear recurrence relations is an informative step towards
alternative combinatorial descriptions of the family of polynomials.

\subsection*{Acknowledgements}
The author would like to thank Jim Haglund for suggesting this problem.
This work has been funded by the Knut and Alice Wallenberg Foundation.

\section{Diagrams and fillings}

A \defin{diagram} $D$ is a subset of $\{(i,j) : i,j \geq 1\}$ which realized as an arrangement of \defin{boxes},
with a box at $(i,j)$ for every $(i,j)$ in $D$. Here, $i$ refers to the \defin{row} and $j$ is the \defin{column}
of box $(i,j)$ and we draw diagrams in the English notation.
For example, $D=\{(1,1),(1,3),(1,4),(2,2),(3,2)\}$ is shown as
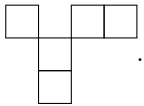
\[
\ytableausetup{centertableaux,boxsize=1.2em}
\begin{ytableau}
 \; & \none & \; & \; \\
 \none  & \; \\
\none  & \; \\
\end{ytableau}.
\]
A filling is said to have $l$ rows, if row $l$ contains a box, but every row below row $l$ is box-free.

Given an integer composition $\alpha = (\alpha_1,\dots,\alpha_l)$, the \defin{diagram of shape $\alpha$}, $D_\alpha$,
is given by
$D_\alpha = \{(i,j) : 1\leq j \leq \alpha_i, 1\leq i \leq l \}.$
If $\beta = (\beta_1,\dots,\beta_l)$ is another integer composition such that $\alpha \supseteq \beta$, that is,
$\alpha_i \geq \beta_i$ for all $i$, then the \defin{diagram of shape $\alpha/\beta$}
is given by the set-theoretical difference $D_\alpha \setminus D_\beta$, and is denoted $D_{\alpha/\beta}$.
Finally, if $D$ is a diagram, let $kD$ be the diagram obtained from $D$ by repeating each column in $D$ $k$ times,
that is,
\[
kD = \bigcup_{(i,j)\in D} \{(i,kj-k+1),(i,kj-k+2),\dots,(i,kj)\}.
\]
Note that $kD_{\alpha/\beta} = D_{k\alpha/k\beta}$.

\subsection{Fillings}

A \defin{filling} of a diagram is a map $T: D \to \setN$, which we represent by writing $T(i,j)$ in the box $(i,j)$.
For example,
\begin{equation}
\ytableausetup{centertableaux,boxsize=1.2em}
\begin{ytableau}
 1 & 8 \\
 \none[\times]  & \none[\times] & 9 & 1 \\
\none[\times]  & \none[\times]  & 5 \\
\none[\times] \\
\none[\times]  &\none[\times]  &\none[\times]  & 9 & 2 & 7 \\
\end{ytableau}\label{eq:exampleFilling}
\end{equation}
is a filling of the diagram with shape $(2,4,3,1,6)/(0,2,2,1,3)$,
where the places marked $\times$ correspond to boxes in $D_\beta$.
The shape of a filling refers to the shape of the underlying diagram.

The $j$th column in a diagram $D$ with $l$ rows has a shape define as the integer composition $(s_1,\dots,s_l)$,
where $s_i = 1$ if $(i,j) \in D$ and $0$ otherwise.
Thus, if $\alpha$ is an integer composition with only $0$ and $1$ as parts,
then the first column of $D_\alpha$ has shape $\alpha$.
Whenever $\alpha$ is an integer partition, $D_\alpha$ is called a Young diagram
and any filling of a Young diagram is called a \defin{tableau}.
A filling with shape $\alpha/\beta$ where both $\alpha$ and $\beta$ are partitions,
is called a \defin{skew tableau}. 

Given a diagram or a filling, we can \defin{duplicate} or \defin{delete} columns.
For example, deleting the fourth column and duplicating the third
column in the filling in \cref{eq:exampleFilling} results in the filling
\begin{equation}
\ytableausetup{centertableaux,boxsize=1.2em}
\begin{ytableau}
 1 & 8 \\
 \none[\times]  & \none[\times] & 9 & 9 & 9\\
\none[\times]  & \none[\times]  & 5 & 5 & 5 \\
\none[\times] \\
\none[\times]  &\none[\times]  &\none[\times] &\none[\times] &\none[\times]   & 2 & 7 \\
\end{ytableau}.
\end{equation}
Note that if the original filling $T$ has shape $D_{\alpha/\beta}$,
then duplication and deletion on $T$ will result in some $T'$ of shape $D_{\alpha'/\beta'}$.
This is straightforward to prove.

\subsection{Column-closed families of fillings}

In most applications, one are interested in a restricted family of fillings, perhaps tableaux or skew tableaux,
together with some conditions on the numbers that appear in the boxes. 
Note that a filling $T$ can be viewed as a concatenation of its columns --- some of which might be empty.
Obviously, $T$ can only be expressed in one such way if we require that the last (rightmost) column is non-empty.

Let $(C_1,\dots,C_l)$ be a filling with columns $C_1,\dots,C_l$ and let $(m_1 C_1,\dots, m_lC_l)$
denote the filling with $m_1$ copies of the column $C_1$, followed by $m_2$ copies of $C_2$ and so on.
Finally, the concatenation, $\sim$, of two fillings $(C_1,\dots,C_l)$ and $(C'_1,\dots,C'_{l'})$
is simply given by
\[
(C_1,\dots,C_l) \sim (C'_1,\dots,C'_{l'}) = (C_1,\dots,C_l,C'_1,\dots,C'_{l'}).
\]

\begin{definition}\label{def:column-closed}

A family of fillings, $\tabFamT$,  is said to be \defin{weakly column-closed} if
\begin{equation}\label{eq:weaklycc}
 (C_1,\dots,C_i,\dots,C_l) \in \tabFamT \text{ if and only if } (m_1 C_1,\dots, m_lC_l) \in \tabFamT
\end{equation}
holds for every combination of integers $m_i$ where $m_i \geq 1$.
The family $\tabFamT$ is said to be \defin{strictly column-closed}
if \cref{eq:weaklycc} holds for every combination where $m_i \geq 0$.
That is, the family is closed under deletion of \emph{any} column.
\end{definition}
Less formally, $\tabFamT$ is weakly column-closed if it is closed under column duplication,
and reduction of \emph{duplicate} columns. The family is strictly column closed if it, in addition,
is closed under removal of \emph{any} column.

\medskip 

Combinatorial objects would not be interested if it weren't for combinatorial statistics.
A \defin{combinatorial statistic} on a family $\tabFamT$ is a map $\stat: \tabFamT \to \setN^s$.
We will study a special type of statistics on fillings:
\begin{definition}\label{def:affine-statistic}
A combinatorial statistic $\stat$ on a weakly column-closed family $\tabFamT$ is \defin{affine} if
\[
\stat(m_1 C_1,\dots, m_lC_l) = A + S_1m_1 + S_2m_2 + \dots + S_lm_l
\]
for all choices of $m_i \geq 1$, where $A$ and $S_i$ are vectors in $\setN^s$.
Similarly, $\sigma$ defined on a strictly column-closed family $\tabFamT$ is \defin{linear} if
\[
\stat(m_1 C_1,\dots, m_lC_l) = S_1m_1 + S_2m_2 + \dots + S_lm_l
\]
for all choices of $m_i \geq 0$.
Note that this is equivalent with the statement that $\stat(T_1 \sim T_2) = \stat(T_1)+ \stat(T_2)$
for every pair $T_1$, $T_2$ of fillings such that $T_1 \sim T_2$ is in $\tabFamT$.
\end{definition}

Note that the statistic given by $w(T)=(w_1,w_2,\dots,w_n)$ where $w_i$ are the number of boxes filled with $i$ in $T$
is a linear statistic. This is usually called the \defin{weight} of $T$.
Finally, two statistics $\stat_1 : \tabFamT \to \setN^{s_1} $ and $\stat_2 : \tabFamT \to \setN^{s_2}$
can be combined into a new statistic $\sigma$ in the obvious manner as $\stat(T) = (\stat_1(T),\stat_2(T))$,
which map to $\setN^{s_1+s_2}$.

\section{Properties of linear recurrences}\label{sec:linearRecurrences}

We first recall some basic notions about linear recurrences.
This can be seen as analogous to the theory of linear differential equations.
\medskip

A sequence $\{a_k(\xvec)\}_{k=0}^\infty$ of functions are said to satisfy a \defin{linear recurrence} of length $r$
if there are functions $c_1(\xvec),\dotsc,c_r(\xvec)$ such that
\begin{equation}\label{eq:linearsequence}
a_{k}(\xvec) +  c_1(\xvec)a_{k-1}(\xvec) + \dotsb + c_r(\xvec)a_{k-r}(\xvec) \equiv 0
\end{equation}
for all integers $k\geq r$. The polynomial (in $t$)
\[
\chi(t) = t^{k} +  c_1(\xvec)t^{k-1} + \dotsb + c_{r-1}(\xvec)t + c_r(\xvec)
\]
is called the \defin{characteristic polynomial} of the recursion.
If the characteristic polynomial factors as $(t-\rho_1)^{m_1}\dotsb (t-\rho_r)^{m_r}$,
where all $\rho_i(\xvec)$ are \emph{different}, then one can express $a_k(\xvec)$ as
\begin{equation}\label{eq:generalForm}
a_k(\xvec) = \sum_{l=1}^r \left(\rho_l(\xvec)\right)^k  \sum_{j=0}^{m_l-1} g_{lj}(\xvec)k^j
\end{equation}
for some functions $g_{li}(\xvec)$, that only depend on the \emph{initial conditions}, that is, the functions $a_0(\xvec)$ to $a_{r-1}(\xvec)$.
In the other direction, any sequence of functions which are of the form given in \cref{eq:generalForm}
satisfy a linear recurrence with $\chi(t)$ as characteristic polynomial.
Notice that the $c_i$ are elementary symmetric polynomials in the $\rho_i$, with some signs.

\medskip 

From now on, we are only concerned about sequences where the $a_k(\xvec)$ and $\rho_j(\xvec)$ are polynomials,
which implies that the $c_i(\xvec)$ are polynomials and the $g_{li}(\xvec)$ are rational functions.
Let $a_k(\xvec)$ and $b_k(\xvec)$ be sequences of polynomials with characteristic polynomials 
given by $\prod_i (t-\rho_i(\xvec))^{p_i}$ and $\prod_i (t-\rho_i(\xvec))^{q_i}$ respectively,
where some of the $p_i$ or $q_i$ may be zero.
Then, as sequences for $k=0,1,\dotsc$,
\begin{itemize}
\item $h(\xvec) a_k(\xvec)$ satisfy the same linear recurrence as $a_k(\xvec)$, where $h(\xvec)$ is any polynomial,
\item $a_k(\xvec) + b_k(\xvec)$ satisfy a linear recurrence with characteristic polynomial given by
\[
 \prod_i (t-\rho_i(\xvec))^{\max(p_i,q_i)}
\]
\item $a_k(\xvec) \cdot b_k(\xvec)$ satisfy a linear recurrence with characteristic polynomial given by
\[
\prod_{\substack{i,j \\ p_i \geq 1\;\; q_j \geq 1}} (t-\rho_i(\xvec)\rho_j(\xvec))^{p_i+q_j-1}.
\]
However, if $\rho_{i_1}(\xvec)\rho_{j_1}(\xvec) = \rho_{i_2}(\xvec)\rho_{j_2}(\xvec)$ for some $(i_1,j_1)\neq(i_2,j_2)$,
some roots of the characteristic equation can be removed --- details are left as an exercise.
As an example:
\[
 a_k(\xvec) = (1+k^3) (5x)^k, \qquad b_k(\xvec) = (2+k^2-k^4) (2x-1)^k
\]
satisfy linear recurrences with characteristic polynomials $(t-5x)^4$ and $(t-(2x-1))^5$ respectively.
The product, $a_k(\xvec)b_k(\xvec) = (1+k^3)(2+k^2-k^4)(10x^2-5x)^k$
satisfy a linear recurrence with characteristic polynomial $(t-(10x^2-5x))^8$.

\item $a_{s k}(\xvec)$ with $s$ a fixed positive integer satisfy a linear recurrence with characteristic polynomial given by
\[
\prod_{i} (t-\rho_i(\xvec)^{s} )^{p_i}.
\]
\end{itemize}
The proofs for these statements follows from writing $a_k(\xvec)$ and $b_k(\xvec)$ in the form \cref{eq:generalForm}
and examining the expressions above.
Note that if $a_k(\xvec)$ and $b_k(\xvec)$ have characteristic polynomials with simple roots,
then so does $h(\xvec)\cdot a_k(\xvec)$, $a_k(\xvec)+b_k(\xvec)$, $a_k(\xvec)\cdot b_k(\xvec)$ and $a_{sk}(\xvec)$.

Finally, the definition of a sequence satisfying a linear recurrence in \cref{eq:linearsequence}
does not provide an easy method to check for a linear recurrence if the $c_i$ are unknown.
A useful shortcut might then be the following observation:
a sequence $\{a_k(\xvec)\}_{k=0}^\infty$ satisfy a linear recurrence of length $r$ if and only if
the following $r\times r$-determinant vanish for all $k\geq r-1$:
\[
\begin{vmatrix}
a_{k} &  a_{k-1} & \dots & a_{k-r+1} \\
a_{k+1} &  a_{k} & \dots & a_{k-r+2} \\
\vdots &  \vdots & \ddots & \vdots  \\
a_{k+r-1} &  a_{k+r-2} & \dots & a_{k}
\end{vmatrix}.
\]
This classical trick can be found in \emph{e.g.} \cite{Lyness1957}.

\subsection{Tableaux and linear recurrences}

\begin{lemma}\label{lem:onetypeColumnRecurrence}
Let $\tabFamT$ be a weakly column-closed family of fillings and
$T=(C_{1},\dots,C_{l})$ is a fixed filling in $\tabFamT$, where no adjacent columns are equal.
Let $\stat : \tabFamT \to \setN^n$ be a linear combinatorial statistic such that
\[
\stat(a_1 C_1,\dots, a_lC_l)  = a_1S_1 + a_2S_2 + \dots + a_lS_l.
\]
Define the sequence of polynomials
\[
F_k(\zvec) = \sum_{\substack{ a_{i} \geq 1 \\  a_{1}+a_{2}+\dots + a_{l} = k }}
\zvec^{\sigma( a_{1}C_{1},\dots, a_{l}C_{l} )} \text{ and } F_{0}(\zvec)=(-1)^{l+1}.
\]
Then $\{ F_{k}(\zvec) \}_{k=0}^\infty$ satisfy a linear recurrence, with characteristic polynomial
\begin{equation}\label{eq:charEqCC}
(t-\zvec^{S_1})(t-\zvec^{S_2})\dotsm(t-\zvec^{S_l}).
\end{equation}
\end{lemma}
\begin{proof}
Note that the definition of $F_k(\zvec)$ implies that $F_k(\zvec)\equiv 0$ whenever $1 \leq k < l$,
and that $F_l(\xvec) = \zvec^{S_1+\dots+S_l}$. These are $l$ conditions and it is easy to see that if we
have a characteristic polynomial of the form \eqref{eq:charEqCC}, then $F_{0}(\zvec)$ must be equal to $(-1)^{l+1}$.

Any tableau of the form $(a_{1}C_{1},\dots, a_{l}C_{l})$ where $a_i \geq 1$ and $a_{1}+a_{2}+\dots + a_{l} > l$,
must have some $a_i \geq 2$. Thus, this tableau can be constructed from some $(a_{1}C_{1},\dots,(a_{i}-1)C_i,\dots, a_{l}C_{l})$ by duplicating column $C_i$.
However, there might be several ways to do this. By using an inclusion-exclusion argument, it is straightforward to show that
\[
F_{k+l}(\zvec) - (\zvec^{S_1}+\dotsb+\zvec^{S_l}) F_{k+l-1}(\zvec) + \dotsb
+(-1)^l(\zvec^{S_1}\dotsm \zvec^{S_l})F_{k}(\zvec) \equiv 0
\]
for all $k\geq 0$. Note that the coefficients are the elementary symmetric polynomials, evaluated at $\zvec^{S_1},\dots,\zvec^{S_l}$,
so factoring the characteristic polynomial gives exactly the expression in \eqref{eq:charEqCC}.
\end{proof}

\begin{lemma}\label{lem:columnsRecurrence}
Let $\tabFamT$ be a weakly column-closed family of fillings, and let the $T(\avec) \in \tabFamT$ be given as
\begin{align*}
T(\avec) &= T_1 \sim (a_{11}C_{11},\dots,a_{1l_1}C_{1l_1}) \sim T_2 \sim  (a_{21}C_{21},\dots,a_{2l_2}C_{2l_2})\sim \dots \\
	&\sim T_m \sim (a_{m1}C_{m1},\dots,a_{m,l_m}C_{m,l_m}) \sim T_{m+1} 
\end{align*}
where each $T_i$ is some fixed (possibly empty) filling and no adjacent columns in each $(C_{i1},\dots,C_{il_i})$ are equal.
Furthermore, let $\stat : \tabFamT \to \setN^n$ be an affine combinatorial statistic, such that
\[
\stat(T(\avec)) = A + a_{11}S_{11} + \dots + a_{m, l_m} S_{m, l_m}.
\]
Let $\alpha=(\alpha_1,\dots,\alpha_m)$ be a fixed integer composition and define the polynomial
\begin{align*}
G_\alpha(\zvec) = \sum_{\substack{ a_{ij} \geq 1 \\  a_{i1}+a_{i2}+\dots + a_{il_i} = \alpha_i }}
\zvec^{\sigma( T(\avec)   )} \text{ and } G_{0}(\zvec) =-(-1)^{l_1 + l_2 + \dots + l_m}\zvec^A
\end{align*}
where the sum is over all $a_{ij}\geq 1$, $1\leq i \leq m$, $1\leq j \leq \alpha_i$ such that $a_{i1}+\dotsb + a_{il_i} = \alpha_i$.
Then
\[
G_{k\alpha}(\zvec) = \zvec^A \prod_{i=1}^m F^{i}_{k \alpha_i}(\zvec)
% \]
\text{ where }
% \[
F^i_k(\zvec) = \!\!\!\! \sum_{\substack{ c_{i} \geq 1 \\  c_{1}+\dotsb + c_{l_i} = k }} \!\!\!\!
\zvec^{ c_{1}S_{i1}+\dots + c_{l_i}S_{il_i} }, \quad F_{0}(\zvec)=-(-1)^{l_i}.
\]
\end{lemma}
\begin{proof}
This is immediate from the the definition of $\sigma$ and the $F^i_k(\zvec)$,
by simply substituting the definition of $F^i_k(\zvec)$ in the product and recognizing the expression for $\sigma$.
\end{proof}

Note that the integer composition $\alpha$ should not be confused with some shape of a tableau.
The composition $\alpha$ rather serves as the number of columns there are in each of the $m$ ``blocks'' of columns in $T(\avec)$.
The functions $G_{k\alpha}(\xvec)$ can now be seen as the generating functions of $\sigma$, as the block sizes grows linearly with $k$,
and each block $i$ consists of column fillings with columns from $\{ C_{i1},\dots, C_{il_i}\}$, each present at least once.
\begin{figure}[!ht]
\centering
\begin{tikzpicture}[scale=0.35,y=-1cm]

\draw[black] (0, 0)--(0,6)--(6,6)--(6,0)--cycle;
\draw[black,dashed] (2, 0)--(2,6);
\draw[black,dashed] (4, 0)--(4,6);
\draw[black] (6, 0)--(6,4)--(10,4)--(10,0)--cycle;
\draw[black,dashed] (8, 0)--(8,4);
\draw[black] (10, 0)--(10,1)--(18,1)--(18,0)--cycle;
\draw[black,dashed] (12, 0)--(12,1);
\draw[black,dashed] (14, 0)--(14,1);
\draw[black,dashed] (16, 0)--(16,1);
\draw[<->] (0 , -0.5)--(5.9, -0.5);
\draw[<->] (6.1 , -0.5)--(9.9, -0.5);
\draw[<->] (10.1 , -0.5)--(18, -0.5);
\node at ( 3 , -1.5) {$\alpha_1$};
\node at ( 8 , -1.5) {$\alpha_2$};
\node at ( 14 , -1.5) {$\alpha_3$};
\end{tikzpicture}
\caption{The role of $\alpha$. Here, all $T_i$ are empty.}
\label{FigKLambda}
\end{figure}
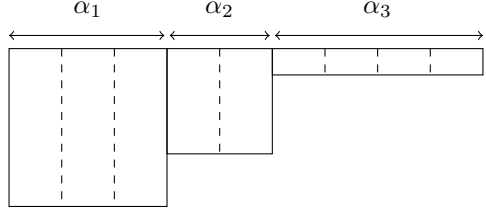
However, note that if all $T_i$ are empty fillings (no columns), then $G_{k\alpha}(\xvec)$
can be seen as generating function for fillings of shape $kD$ for some \emph{fixed} diagram $D$ as in \cref{FigKLambda}.
In the proof of \cref{prop:mainProp}, the relation between $\alpha$ and $D$ is explained in more detail.

\begin{corollary}\label{cor:linearRecurrence}
$\{ G_{k \alpha}(\zvec) \}_{k=0}^\infty$ satisfy a linear recurrence with characteristic polynomial given by
\begin{equation}\label{eq:mainCharPoly}
 \prod_{\substack{ 1\leq j_1 \leq l_1 \\ 1\leq j_2 \leq l_2  \\ \vdots \\ 1\leq j_m \leq l_m \\ }}
\left( t - \zvec^{\alpha_1 S_{1j_1}}   \zvec^{\alpha_2 S_{2j_2}} \dotsm \zvec^{\alpha_m S_{m,j_m}}  \right).
\end{equation}
Furthermore, if $\sigma$ is linear, then \cref{eq:mainCharPoly} can be expressed as
\begin{equation}\label{eq:mainCharPolyLinear}
 \prod_{\substack{ 1\leq j_1 \leq l_1 \\ 1\leq j_2 \leq l_2  \\ \vdots \\ 1\leq j_m \leq l_m \\ }}
\left( t - \zvec^{\sigma( \alpha_1 C_{1j_1}, \alpha_2 C_{2j_2}, \dotsc, \alpha_m C_{m,j_m})}  \right).
\end{equation}
Multiple roots can be disregarded if $\{S_{i1},\dots,S_{il_i}\}$ are all distinct for every $i$.
\end{corollary}
\begin{proof}
Each $F^i_k(\zvec)$ can be seen as generated by a linear statistic $\sigma'(T)=\sigma(T)-A$.
Therefore, each of these satisfy a linear recurrence, according to \cref{lem:onetypeColumnRecurrence}.
The theory of linear recurrences now imply that the $G_{k \alpha}(\zvec)$ also does,
with a characteristic polynomial as described above, since $G_\alpha$ is essentially a product of the $F^i$.

\cref{eq:mainCharPolyLinear} follows from linearity of $\sigma$ together with the definition of $\sigma$.
Note that the value of $\sigma( \alpha_1 C_{1j_1}, \alpha_2 C_{2j_2}, \dotsc, \alpha_m C_{m,j_m})$ is defined by linearity $\sigma$,
but that the tableau we evaluate on might not be in $\tabFamT$ (if some $\alpha_i=0$) unless this family is strictly column closed.
The statement about simple roots follows immediately from the theory of linear recurrences.
\end{proof}

\medskip

So far, we have only treated generating functions of subsets of tableaux where the columns
are from a specified subset and each column appear at least once. We will now treat the case where only the
family of fillings and the diagram shape defines the generating function.
To do that, it is natural to restrict ourself to a special type of families of fillings.

A family $\tabFamT$ is said to be \defin{well-behaved} if every filling $T \in \tabFamT$
satisfies the following two properties:
\begin{itemize}
\item if two columns in $T$ are identical, then all columns in between are also identical to these two.
\item if two columns $C_1$ and $C_2$ are different and $C_1$ appears to somewhere the left $C_2$, then
$C_1$ never appears to the right of $C_2$ in some other filling $T' \in \tabFamT$.
\end{itemize}
For example, fillings such that every row is weakly decreasing (or increasing) are well-behaved.

\begin{proposition}\label{prop:mainProp}
Let $\tabFamT$ be well-behaved  weakly column-closed family of fillings and let $\sigma$ be an affine statistic defined on $\tabFamT$.
Let $D$ be a fixed diagram and define the polynomials $H_{D}(\zvec)$ as
\[
H_{D}(\zvec) = \sum_{ T \in \tabFamT(D,n)} \zvec^{\stat(T)}
\]
where $\tabFamT(D,n)$ is the set of all fillings in $\tabFamT$ with shape $D$ and for every box $(i,j)$ in such a filling,
we have $1\leq T(i,j)\leq n$. Then $\{ H_{kD}(\zvec) \}_{k=1}^\infty$ satisfy a linear recurrence.
Furthermore, if $\sigma$ is linear, then the characteristic polynomial of the recurrence is given by
\begin{equation}\label{eq:mainCharPolyLinear2}
\prod_{ T } \left( t-\zvec^{\stat(T)} \right)
\end{equation}
where $T$ runs over all tableaux of shape $D$ such that any adjacent columns of same shape in $T$ are identical,
and each $T$ can be obtained from some $\tabFamT(kD,n)$ by deleting some columns.
\emph{Note that $T$ might \emph{not} itself be an element in $\tabFamT$.}
However, if $\tabFamT$ is strictly column closed, then each such $T$ is in $\tabFamT(D,n)$.
\end{proposition}
\begin{proof}
Note that every column in $kD$ has the same shape as some column in $D$.
Since we may only fill boxes with entries from $[n]$, there is a finite number columns that
can appear in $\tabFamT(kD,n)$.
Furthermore, since $\tabFamT$ is well-behaved, there is a \emph{finite} number of lists of columns,
$(C_1,C_2,\dots,C_l)$, such that all $C_i$ are different and $C_i$ never appears to the right of $C_j$ in
some filling in  $\tabFamT$, whenever $i<j$.
Thus, for every $k$, every filling in $\tabFamT(kD,n)$ can be obtained in a unique way from such a list,
by duplicating some columns in that list.
Hence, $H_{kD}(\zvec)$ can be expressed as a finite sum over such lists $(C_1,C_2,\dots,C_l)$,
where each term corresponds to fillings $T$ of shape $kD$ where each column in $T$ is in the list
and every column in the list appears at least once in $T$.

More specifically, let the diagram $D$ be expressed as the concatenation
$ D = (\alpha_1 s_1 , \alpha_2 s_2, \dotsc, \alpha_m s_m) $
where the $s_i$ are column shapes and $s_i \neq s_{i+1}$ (here, we use the same notation as for filled columns).
Then every filling in $\tabFamT(kD,n)$ can be obtained in a unique way as
\[
(a_{11}C_{11},\dots,a_{1l_1}C_{1l_1}) \sim  (a_{21}C_{21},\dots,a_{2l_2}C_{2l_2})\sim \dots \sim  (a_{m1}C_{m1},\dots,a_{m,l_m}C_{m,l_m})
\]
where each $C_{ij}$ has shape $s_i$ and $a_{i1}+\dots+a_{i,l_i}=\alpha_i$.
Hence, $H_{kD}(\xvec,\tvec)$ can be expressed as a sum over polynomials of the same form as $G_{k\alpha}$
in \cref{lem:columnsRecurrence}.
\cref{cor:linearRecurrence} tells us that each $G_{k\alpha}$ satisfy a linear recurrence, so the sum of such
sequences must too. This proves the first statement in the proposition.

The second statement follows from \cref{cor:linearRecurrence} and observing that
the $S_{ij}$ in \cref{eq:mainCharPoly} can be replaced by $\stat(C_{ij})$, since $\stat$ is linear.
The observation that it is enough to only consider tableaux where adjacent columns of same shape have identical fillings
is a consequence of the combinatorial interpretation of the $F_{k\alpha_i}^i(\zvec)$ in \cref{lem:columnsRecurrence} --- a block of size $k \alpha_i$
must have $\alpha_i$ copies of some column if $k$ is sufficiently large and now a similar inclusion-exclusion reasoning apply
as in \cref{lem:onetypeColumnRecurrence}.
\end{proof}

\begin{corollary}\label{cor:simpleRoots}
Let $(\stat,\statT)$ be an affine statistic,
such that the restriction to $\stat$ is linear and $\stat(C_1) \neq \stat(C_2) \Rightarrow C_1 \neq C_2$
for any pair of columns that appear in a filling in $\tabFamT$.
Then the characteristic polynomial in \eqref{eq:mainCharPolyLinear2} can be taken to have only simple roots.
\end{corollary}
\begin{proof}
The fact that $\stat(C_1) \neq \stat(C_2) \Rightarrow C_1 \neq C_2$ implies that
the values of the $S_{ij}$ in \eqref{eq:mainCharPoly} are all distinct.
Since $\{F_{kD}(\zvec)\}_{k=1}^\infty$ can be expressed as a sum of sequences,
each of which has simple roots in its characteristic polynomial,
the statement follows by using the theory in \cref{sec:linearRecurrences}.
\end{proof}

\section{Augmented fillings}

This section introduce the diagram fillings that are responsible for key polynomials, Demazure atoms and Hall--Littlewood polynomials.
We follow the terminology in \cite{Haglund2011463,Mason2009}, with a few minor modifications.

Let $\beta = (\beta_1,\dots,\beta_n)$ be a list of $n$ different positive integers
and let $\alpha=(\alpha_1,\dots,\alpha_n)$ be a weak integer composition, that is, a vector with non-negative integer entries.
An \defin{augmented filling} of shape $\alpha$ and \emph{basement} $\beta$
is a filling of a Young diagram of shape $(\alpha_1,\dotsc,\alpha_n)$ with positive integers,
augmented with a zeroth column filled from top to bottom with $\beta_1,\dotsc,\beta_n$.

\begin{definition}
Let $T$ be an augmented filling. Two boxes $a$, $b$, are \defin{attacking}
if $T(a)=T(b)$ and the boxes are either in the same column, or they are in adjacent columns, with the rightmost box in a row strictly below the other box.
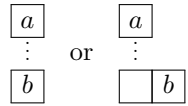
\begin{align*}
\begin{ytableau}
a  \\
\none[\scriptstyle\vdots] \\
b  \\
\end{ytableau}\quad
\text{or}\quad
\begin{ytableau}
 a & \none \\
 \none[\scriptstyle\vdots] \\
  & b \\
\end{ytableau}
\end{align*}
\end{definition}
A filling is \defin{non-attacking} if there are no attacking pairs of boxes.

\begin{definition}
Let $T$ be an augmented filling with weakly decreasing rows.
An \defin{inversion triple of type $A$} is an arrangement of boxes, $a$, $b$, $c$,
located such that $a$ is immediately to the left of $b$, and $c$ is somewhere below $b$,
and the row containing $a$ and $b$ is at least as long as the row containing $c$
and $T(a) \geq T(c) \geq T(b)$.

Similarly, an \defin{inversion triple of type $B$} is an arrangement of boxes, $a$, $b$, $c$,
located such that $a$ is immediately to the left of $b$, and $c$ is somewhere above $a$,
and the row containing $a$ and $b$ is \emph{strictly} longer than the row containing $c$
and $T(a) \geq T(c) \geq T(b)$.
\begin{equation}\label{eq:invTriplets}
\text{Type $A$:}\quad
\ytableausetup{centertableaux,boxsize=1.2em}
\begin{ytableau}
 a & b \\
 \none  & \none[\scriptstyle\vdots] \\
\none & c \\
\end{ytableau}
\qquad
\text{Type $B$:}\quad
\ytableausetup{centertableaux,boxsize=1.2em}
\begin{ytableau}
 c & \none \\
\none[\scriptstyle\vdots]  & \none \\
a & b \\
\end{ytableau}
\end{equation}
\end{definition}
\textbf{Warning!} This definition slightly different from what is stated in \cite{Haglund2011463}.
However, the definitions coincides whenever the rows in the filling are weakly decreasing and we are only concerned with that
special case.

\begin{definition}
A \defin{semi-standard augmented filling}, (\ssaf) of \emph{shape} $\alpha$ and \emph{basement} $\beta$
is an augmented filling of shape $\alpha$ and basement $\beta$ with weakly decreasing rows and without any inversion triples.
\end{definition}
Note that this definition implies that there are no attacking boxes in an \ssaf.
In particular, all entries in every column are different.

\begin{example}\label{ex:generalAtomFilling}
Here is an example of a semi-standard augmented filling, with basement $(1,3,2,5,4)$.
\[
\ytableausetup{centertableaux,boxsize=1.2em}
\begin{ytableau}
\mathbf{1} & \none \\
\mathbf{3} & 3 & \underline{1} & 1 \\
\mathbf{2} & 2 \\
\mathbf{5} & 5 & 5 & 5\\
\mathbf{4} & 4 & \underline{4} & \underline{3} & 2
\end{ytableau}
\]
We can for example check the underlined entries for the type $B$ inversion triple condition --- since $4 \leq 1 \leq 3$ is \emph{not} true,
they do not form such a triple. It is left as an exercise to check that no other triples are inversion triples.
\end{example}

\begin{lemma}\label{lem:ssafIsCCandW}
The semi-standard augmented fillings is a weakly column-closed, well-behaved family. 
\end{lemma}
\begin{proof}
It suffices to show that duplication of a column in a $\ssaf$ $T$
do not introduce any inversion triples and it is enough to check that there are no inversion triples in adjacent and identical columns.

Assume that $a$, $b$, $c$ form an inversion triple, as in \cref{eq:invTriplets} (either type).
Since the columns are identical, $T(a)=T(b)$ which implies $T(a)=T(c)=T(b)$.
However, this is means that two boxes in the same column are identical, so they are attacking.
This contradicts the fact that the filling is a $\ssaf$.
\end{proof}

\medskip

Let $\ssaf(\beta,\alpha)$ be the set of all semi-standard augmented fillings with basement $\beta$ and shape $\alpha$.
Note that $\ssaf(\beta,\alpha)$ is a finite set.
Given an augmented filling $T$, let $w(T) = (w_1,\dots,w_n)$ where $w_i$ count the number of boxes with content $i$
\emph{not including the basement}. 
The \defin{generalized Demazure atoms} are defined as
\begin{equation}\label{eq:atomPolyDef}
\atom_{\beta,\alpha}(\xvec) = \sum_{T \in \ssaf(\beta,\alpha)} \xvec^{w(T)}.
\end{equation}
The special case when $\beta_i = i$ corresponds to the ordinary Demazure atoms,
introduced by Lascoux and Schützenberger in \cite{Lascoux1990Keys} under the name \emph{standard bases}.

Let $\nawf(\alpha)$ denote the set of all non-attacking augmented fillings of shape $\alpha$ with weakly decreasing rows
and basement given by $\beta_i =i$.
The non-symmetric, integral form Hall--Littlewood polynomials, $E_\alpha(x_1,\dots,x_n)$, may be defined as
\begin{equation}\label{eq:HLPolyDef}
E_\alpha(\xvec;t) = \sum_{T \in \nawf(\alpha) } \xvec^{w(T)} t^{\coinv(T)} (1-t)^{\dn(T)}
\end{equation}
where $\coinv(T)$ is the number of triples in $T$ which are \emph{not} inversion triples,
and $\dn(T)$ is the number of pairs of adjacent boxes, $(i,j)$ and $(i,j+1)$ such that $T(i,j) \neq T(i,j+1)$ (different neighbors).
This formula was first given in \cite{Haglund2011463}.
It is straightforward to show that the $\nawf$  form a weakly column-closed and well-behaved family.
They show that the ordinary Hall--Littlewood polynomials $P_\mu(\xvec;t)$ can be expressed
as
\begin{equation}\label{eq:nonSymHLinHL}
 P_\mu(\xvec;t) = \sum_{\substack{\gamma \\ \mu = \lambda(\gamma) }} E_\gamma(\xvec;t)
\end{equation}
where $\lambda(\gamma)$ is the unique integer partition that is obtained from
the weak integer composition $\gamma$ by sorting the parts in decreasing order.

\begin{lemma}
The statistics $\dn$ and $\coinv$ are affine statistics.
\end{lemma}
\begin{proof}
It follows immediately from the definition of $\dn$
that if $\dn(C_1,\dots, C_l) = A$, then $\dn(m_1 C_1,\dots, m_lC_l) = A$ for all $m_i \geq 1$,
so this is affine.
The fact that $\coinv$ is affine is also quite straightforward and is left as an exercise to the reader.
\end{proof}

Using \cref{prop:mainProp}, \cref{lem:ssafIsCCandW} and \cref{cor:simpleRoots}, we have the following result:
\begin{corollary}
The sequences $\atom_{\beta,k\alpha}(\xvec)$ and $E_{k\alpha}(\xvec;t)$ for $k=1,2,\dotsc$ satisfy linear recurrences with simple roots.
\end{corollary}

Note that \cref{eq:nonSymHLinHL} implies that the ordinary Hall--Littlewood $P$ polynomials satisfy a linear recurrence.
These are usually (see \cite{Macdonald1995}) defined as
\begin{equation}\label{eq:HallLittlewoodPDefinition}
P_\lambda(\xvec;t) = \left( \prod_{i\geq 0} \prod_{j=1}^{m_\lambda(i)} \frac{1-t}{1-t^{j}} \right)
\sum_{\sigma \in \symS_n} \sigma\left(   x_1^{\lambda_1}\dotsm x_n^{\lambda_n} \prod_{i<j} \frac{x_i - tx_j}{x_i-x_j}     \right),
\end{equation}
where $\lambda = (\lambda_1,\dotsc,\lambda_n)$, some parts might be zero and $m_\lambda(i)$ denotes the number of parts of $\lambda$ equal to $i$,
and $\sigma$ act on the indexing of the variables.

Observe that from this definition, it is quite clear that $\{P_{k\lambda}(\xvec;t)\}_{k=1}^\infty$ satisfies a linear recurrence,
since for a fixed $\sigma$ in \eqref{eq:HallLittlewoodPDefinition}, the expression is of the form $g(\xvec;t) \sigma(\xvec)^{\lambda}$
where $g$ is independent under $\lambda \mapsto k\lambda$. Now compare with \cref{eq:generalForm} above.

\subsection{Key tableaux and key polynomials}

Let $\alpha = (\alpha_1,\dots,\alpha_n)$ be a weak integer composition.
To any such composition, construct a composition with unique entries, $\beta$, and a partition $\lambda$
as follows: Create an augmented Young diagram with shape $\alpha$ and fill the zeroth column with the numbers $1,\dots,n$ in decreasing order.
Remove all rows for which $\alpha_i=0$ and sort the remaining rows according to the number of boxes, in a decreasing manner.
If two rows has the same number of boxes, preserve the relative order\footnote{Although, in this paper, this convention will not be important.}.
The resulting diagram has the shape of a partition, $\lambda$, which we denote $\lambda(\alpha)$,
and the zeroth column will be our basement $\beta(\alpha)$.
It is easy to show that this process can be reversed, that is, to any pair $(\beta,\lambda)$, there is a corresponding $\alpha$.
Finally, note that $\beta(k\alpha) = \beta(\alpha)$ and $\lambda(k\alpha) = k \lambda(\alpha)$ for non-negative integers $k$.

This correspondence is illustrated in \cref{eq:alphabetalambda}, for $\alpha=(0,2,3,4,2,0,1)$
and the tuple $\beta=(4,5,6,3,1)$, $\lambda=(4,3,2,2,1)$.
\begin{equation}\label{eq:alphabetalambda}
\ytableausetup{centertableaux,boxsize=1.2em}
\begin{ytableau}
7             \\
6 &   &       \\
5 &   &   &   \\
4 &   &   &  &\\
3 &   &       \\
2             \\
1 &
\end{ytableau}
\quad \longleftrightarrow \quad
\begin{ytableau}
4 &   &   &  &\\
5 &   &   &   \\
6 &   &       \\
3 &   &       \\
1 &
\end{ytableau}
\end{equation}

The \defin{key polynomials} generalizes the Schur polynomials and are indexed by integer compositions.
They can be defined as
\begin{equation}\label{eq:keyPolyDef}
\key_\alpha(\xvec) = \sum_{T \in \ssaf(\beta(\alpha),\lambda(\alpha))} x^{w(T)}.
\end{equation}
Note that the key polynomials are a subset of the generalized Demazure atoms.
This motivates the definition of a \defin{key tableau} as a semi-standard augmented filling of partition shape,
and we let $\keytab(\beta,\lambda) = \ssaf(\beta,\lambda)$ to emphasize that this subset is of special interest.
Note that we only need to be concerned about inversion triples of type $A$ since we are dealing with partition shapes.
\medskip

Given a column $(\beta_1,\dots,\beta_n)$ and a set of entries $\{c_1,\dots,c_n\}$,
there is at most one one way to arrange the entries $c_1,\dots,c_n$ in a column next to $\beta$ such that the result fulfills all properties of a key tableau.
First of all, it is easy to see that a necessary condition is that $\beta_i \geq c_i$ for all $i$, for some enumeration of the $c_i$,
in order to have decreasing rows in the result.

Secondly, the lack of inversion triples in a key tableau implies that
the order of the entries in the second column is unique; if $(a,b,c)$
is an inversion triple of type $A$, then transposing the entries in box $b$ and $c$ yield a non-inversion triple.
This defines a total order among the elements in the second column,
so there can be at most one filling where the second
column (as a set) consists of the entries $c_1,\dots,c_n$.

The following lemma shows that under the obvious condition that if $\beta_i \leq c_i$ for all $i$, then the $c_i$ can be arranged in
such a way that the columns form a proper key tableau with basement $\beta$.
We prove a stronger statement:

\begin{lemma}\label{lem:sortingToKey}
Let $T$ be a Young diagram of partition shape $1+\lambda$, filled with positive integers in such a way that each row is weakly decreasing,
and each column contains unique entries, and the first column is given by $\beta$.
Then each column in $T$ can be sorted in a unique way such that the result is a key tableau of with shape $\lambda$ and basement $\beta$.
\end{lemma}
\begin{proof}
We do the proof in several steps. The first case we cover is the case when all parts in $\lambda$ has the same size, that is,
all columns of $T$ has the same height. It is easy to see that in this case, we only need to show the statement for two columns.
Thus, assume $(\beta_1,\dots,\beta_n)$ and $(c_1,\dots,c_n)$ are given, with $\beta_i \geq c_i$ for all $i$.

We now perform the following sorting procedure on the second column.
Let $c_i$ be the largest entry in the second column such that $\beta_1 \geq c_i$, and transpose $c_1$ and $c_i$.
Since $c_i \geq c_1$, the rows are still weakly decreasing after this transposition.
Note that $\beta_1$ and $c_i$ cannot be involved in an inversion triple later on:
if there is some $c_j$ such that $\beta_1 \geq c_j > c_i$, this would violate the maximality of the choice of $c_i$.

We now proceed recursively on the remaining entries of the two columns, $(\beta_2,\dots,\beta_n)$
and $(c_2,\dots,c_n)$ where we have performed a transposition in the second column.

\medskip

To handle tableaux with more than two columns, simply apply the permutation
that takes the original column $c$ to the result on all subsequent columns.
The result will now still be a tableau with weakly decreasing rows,
but the first two columns do not contain any inversion triples.
Proceed with the same method on column $2$ and $3$, then $3$ and $4$, and so on.

\medskip

Note that if $c_1 < c_2 < \dotsb < c_r \leq \beta_i$ for all $i$, and $c_r < c_j$ for all $j>r$,
then the second column after the above procedure will end in the sequence $c_r,c_{r-1},\dotsc, c_1$, reading from top to bottom.
Thus, to turn an arbitrarily shaped tableau $T$ into a key tableau, we first augment each column
with negative integers, such that all columns have the same height, and a new entry on row $i$ will get the value $-i$.
After performing the sorting procedure, the above observation implies that we can remove all
boxes with negative entries from the result and recover a key tableau with the same shape as $T$.
\end{proof}

Note that \cref{lem:sortingToKey} implies that if $T$ is a key tableau, then one can remove any column from it,
reorder the entries in each column and obtain another key tableau.
In some sense, key tableaux behave similarly to a strictly column-closed family of tableaux.

\begin{example}
Here we illustrate the sorting procedure described in \cref{lem:sortingToKey}.
We start with the tableau $T$ which is then augmented with negative integers.
\[
T=
\ytableausetup{centertableaux,boxsize=1.3em}
\begin{ytableau}
8 & 5 & 4 & 1 \\
4 & 3 & 2 & 2 \\
6 & 6 & 5     \\
7 & 4         \\
\end{ytableau}
\quad\quad
\longrightarrow
\quad\quad
\begin{ytableau}
8 & 5 & 4 & 1 \\
4 & 3 & 2 & 2 \\
6 & 6 & 5 &-3 \\
7 & 4 &-4 &-4 \\
\end{ytableau}.
\]
The second column is then sorted and the entries in the other columns are permuted in the same fashion.
Two more steps are performed to sort the remaining two columns.
\[
\begin{ytableau}
8 & 6 & 5 &-3 \\
4 & 4 &-4 &-4 \\
6 & 5 & 4 & 1 \\
7 & 3 & 2 & 2 \\
\end{ytableau}
\quad
\stackrel{\text{3rd column}}{\longrightarrow}
\quad
\begin{ytableau}
8 & 6 & 5 &-3 \\
4 & 4 & 4 & 1 \\
6 & 5 & 2 & 2 \\
7 & 3 &-4 &-4 \\
\end{ytableau}
\quad
\stackrel{\text{4th column}}{\longrightarrow}
\quad
\begin{ytableau}
8 & 6 & 5 & 2 \\
4 & 4 & 4 & 1 \\
6 & 5 & 2 &-3 \\
7 & 3 &-4 &-4 \\
\end{ytableau}
\]
Removing the boxes with negative entries now yield a proper key tableaux with the same shape and basement as $T$.
\end{example}

\begin{remark}
Note that \cref{lem:sortingToKey} does not generalize to arbitrary semi-standard augmented fillings.
For example, it is impossible to remove the first column in \cref{ex:generalAtomFilling}
and reorder the entries in the remaining non-basement columns into a valid SSAF --- the $1$s always appear in
some attacking configuration. 
\end{remark}

\subsection{Key polynomial recurrence}

We are now ready to state one of the main result of this paper.

\begin{theorem}\label{thm:Main}
The sequence of polynomials $\{\key_{k \alpha}(\xvec) \}_{k=1}^\infty$ satisfy a linear recurrence
with
\[
\prod_{T} (t - \xvec^T)
\]
as characteristic polynomial, where the product is taken over all key tableaux of shape $\alpha$
such that columns of equal height have the same filling and multiple roots in the product are ignored.
\end{theorem}
\begin{proof}
This follows almost immediately from \cref{prop:mainProp}, except that $\keytab$
is not a strictly column-closed family. However, \cref{lem:sortingToKey}
implies that the tableaux that appear in the product \cref{eq:mainCharPolyLinear2},
can be rearranged to key tableaux, while preserving the weight of the tableau.
\end{proof}

\section{The polytope side}

In this section, we show that the integer point transform of polytopes with
the \defin{integer decomposition property}, (IDP), satisfy a linear recurrence.
In particular, this can be used to give an alternate proof of \cref{thm:Main}.

An \emph{integral polytope} is the convex hull of a finite set of integer points in $\setR^d$.
The $k$-dilation of a polytope $\polyP$ is defined as $k \polyP = \{k \xvec : \xvec \in \polyP\}$
where $k$ is a non-negative integer, and it is easy to see that this is an integral polytope if $\polyP$ is.
Furthermore, a polytope $\polyP$ is said to have the \defin{integer decomposition property}
if for every integer $k\geq 1$, every lattice point $\xvec \in k\polyP \cap \setZ^d$ can be expressed
as $\xvec = \xvec_1 + \dots + \xvec_k$ with $\xvec_i \in \polyP$.
Note that only integral polytopes can have the integer decomposition property
and that every face of a polytope with IDP is also a polytope with the IDP.

The following proposition shows that certain polynomials obtained from polytopes
satisfies a linear recurrence. The argument is very similar to that in \cref{lem:onetypeColumnRecurrence}.
\begin{proposition}\label{prop:recurrenceFromIDP}
Let $\polyP$ be an integrally closed polytope in $\setR^d$ and let $p_k(\zvec)$ be the polynomial defined as
\begin{equation} \label{eq:integerPointTransform}
p_k(\zvec) = \sum_{\xvec \in k\polyP \cap \setZ^d} \zvec^\xvec.
\end{equation}
Then the sequence $p_j(\zvec)$ for $j=0,1,\dots$ satisfies a linear recurrence with characteristic polynomial
given by
\[
\prod_{\xvec \in \polyP \cap \setZ^d} (t-\zvec^\xvec).
\]
\end{proposition}
\begin{proof}
Since $\polyP$ has the IDP, one can easily show that
every lattice point in $k\polyP$ can be expressed as a sum of
a lattice point in $(k-1)\polyP$ plus a lattice point in $\polyP$.
Therefore
\[
p_k(\zvec) - \left( \sum_{\xvec \in \polyP \cap \setZ^d} \zvec^\xvec  \right) p_{k-1}(\tvec)
\]
is a polynomial with only \emph{negative} coefficients corresponding to points in $k\polyP$
that are expressible in more than one way as $\xvec + \yvec$
with $\xvec$ in $(k-1)\polyP \cap  \setZ^d$ and $\yvec$ in $\polyP \cap  \setZ^d$.
Hene
\[
p_k(\zvec) - \left( \sum_{\xvec \in \polyP \cap \setZ^d} \zvec^\xvec  \right) p_{k-1}(\zvec)
+ \left( \sum_{\xvec \neq \yvec \in \polyP \cap \setZ^d} \zvec^\xvec \cdot \zvec^\yvec  \right) p_{k-2}(\zvec)
\]
is again a polynomial with \emph{positive} coefficient corresponding to
lattice points in $k\polyP$ expressible in at least three different ways.
Repeating this argument using the principle of inclusion-exclusion then yield the desired formula.
\end{proof}

The polynomial defined in \cref{eq:integerPointTransform} for $k=1$ is
commonly known as the \defin{integer-point transform} of $\polyP$.

The intersection of two faces of a polytope is also a face (of possibly lower dimension) of the polytope.
This enables us to generalize \cref{prop:recurrenceFromIDP} slightly:
\begin{corollary}\label{cor:integerPointTransformFaces}
Let $\polyP$ be a polytope with the integer decomposition property
and let $F_1,\dots,F_l$ be faces of $\polyP$. Define the sequence of polynomials
\begin{equation} \label{eq:integerPointTransform2}
p_k(\zvec) = \sum_{\xvec \in k(\cup_i F_i) \cap \setZ^d} \zvec^\xvec.
\end{equation}
Then the sequence $p_j(\zvec)$ for $j=0,1,\dots$ satisfies a linear recurrence with characteristic polynomial
given by
\[
\prod_{\xvec \in (\cup_i F_i) \cap \setZ^d} (t-\zvec^\xvec).
\]
\end{corollary}
\begin{proof}
The integer point transform of each face $F_i$ satisfy a linear recurrence under dilation by $k$,
and the polynomials in \cref{eq:integerPointTransform2} can be expressed (via inclusion-exclusion)
as a linear combination of such integer point transforms of faces.
\end{proof}

In \cite{Kiritchenko2010}, it was proven that key polynomials (Demazure characters)
can be expressed as (a certain specialization of) the integer point transform of a
union of faces of a Gelfand--Tsetlin polytope.
Such polytopes are known to have the integer decomposition property, see \emph{e.g.} \cite{Alexandersson20141},
so \cref{cor:integerPointTransformFaces} implies a weaker version of \cref{thm:Main}.
Note that the linear recurrences allow us to define $\key_{0\alpha}(\xvec)$ and it follows from the polyhedral complex interpretation
that this always evaluates to $1$ (there is exactly one lattice point in the union of faces with dilation $0$, namely the origin).
It would be interesting to see a direct proof of this fact without using the polytope interpretation.
\medskip

After extensive computer experimentation, it is hard not to ask the following question:
\begin{question}
Does the polynomial $k \mapsto \key_{k\alpha}(1^n)$ always have non-negative coefficients?
\end{question}
The case when $\alpha$ is a partition corresponds to a Schur polynomial and it is known that
\[
 \schur_\lambda(1^n) = \prod_{1\leq i<j \leq n} \frac{\lambda_i-\lambda_j + j-i}{j-i}
\]
where $n$ is the number of variables. This gives a positive answer to the question in this case.

\section{The operator side}

Some families of polynomials, such as the Schubert and key polynomials can be defined via \defin{divided difference operators}.
Let $s_i$ denote the transposition $(i,i+1)$ and let such transpositions act on $\setZ[x_1,\dotsc,]$ by permuting the indices of the variables.
Define the divided difference operators
\[
 \partial_i = \frac{1-s_i}{x_i-x_{i+1}}, \quad \pi_i = \partial_i x_i.
\]
Given a permutation $\omega \in \symS_n$, it can be expressed as a product of transpositions,
$\omega = s_{i_1}\dotsm s_{i_l}$. When the length $l$ is minimal, we say that $i_1i_2\dots i_l$ is a \defin{reduced word} of $\omega$.
Then, let $\partial_\pi = \partial_{i_1}\dotsm \partial_{i_l}$ and $\pi_\omega = \pi_{i_1}\dotsm \pi_{i_l}$.
It can be shown that these operators does not depend on the choice of the reduced word.

The key polynomials may now be defined \cite{ReinerShimozono1995} as $\key_\alpha(\xvec)  = \pi_{u(\alpha)} x^{\lambda(\alpha)}$,
where $\lambda(\alpha)$ is the partition obtained by sorting the parts of $\alpha$ in decreasing order
and $u(\alpha)$ is a permutation that sorts $\alpha$ into a partition shape.
That this indeed is equivalent to the definition above was proved in \cite{Mason2009}.
We will now give yet another proof that the key polynomials satisfy linear recurrences.
First, note that $\xvec^{k\lambda}$ is a geometric series as $k=0,1,\dotsc$ and thus satisfy a linear recurrence
with characteristic polynomial $t-\xvec^\lambda$. Now note that if $\{f_k(\xvec)\}_{k=0}^\infty$ satisfy a linear recurrence,
then so does $\{\partial_i f_k(\xvec)\}_{k=0}^\infty$ and $\{\pi_i f_k(\xvec)\}_{k=0}^\infty$. The result is now a consequence of induction.

The \defin{Schubert polynomials}, $\schubert_\omega(\xvec)$, indexed by permutations in $\symS_n$, are defined in a similar fashion,
\[
 \schubert_\omega(\xvec) =  \partial_{(w^{-1}w_0)} x_1^{n-1}x_2^{n-2}\dotsm x_{n-1}^{1}
\]
where $\omega_0$ is the longest permutation in $\symS_n$, namely $(n,n-1,\dotsc,1)$ in one-line notation.
Using a similar reasoning as for the key polynomials, one can produce sequences of Schubert polynomials that satisfies
linear recurrences.

\section{Appendix: Some families of column-closed fillings}

In this section, we review some common families of
column-closed fillings, related combinatorial statistics and generating functions over subsets of such families.
Some statements here are well-known or very easy to show, so we present them without proof.
\cref{prop:mainProp} implies that all these polynomials we define below satisfy linear recurrences.

\subsection{Flagged skew semi-standard Young tableaux}

Let $\lambda$ and $\mu$ be partitions with at most $l$ parts, such that $\lambda \supseteq \mu$.
Let $\ssyt(\lambda/\mu,n)$ be the set of fillings of $D_{\lambda/\mu}$ with entries in $[n]$,
such that each row is weakly increasing and each column is strictly increasing.
Then for every $l$ and $n$, the families
\[
\bigcup_{\lambda \supseteq \mu} \ssyt(\lambda/\mu,n) \text{ and  } \bigcup_{\lambda} \ssyt(\lambda,n)
\]
are strictly column-closed families, where the unions are taken over shapes
with at most $l$ rows.
On any filling $T$ with entries in $[n]$, we define the statistic $w(T):\tabFamT \to \setN^n$
such that if $w(T) = (w_1,\dots,w_n)$, then $w_i$ is the number of boxes in $T$ filled with $i$.
It is evident that $w$ is a linear statistic.

Finally, the skew Schur polynomials in $n$ variables, indexed by skew partition shapes $\lambda/\mu$, are defined as
\[
 \schur_{\lambda/\mu}(\xvec) = \sum_{T \in \ssyt(\lambda/\mu,n)} \xvec^{w(T)}.
\]

\bigskip

Even more general, let $\lambda \supseteq \mu$ be a shapes with at most $l$ rows
and let $a$ and $b$ be increasing sequences of integers of length $l$, such that $a_i \leq b_i$ for all $i$.
Let $\ssyt(\lambda/\mu,a,b,n) \subseteq \ssyt(\lambda/\mu,n)$ be the subset of fillings $T$,
such that $a_i \leq T(i,j) \leq b_i$ for every box $(i,j) \in D_{\lambda/\mu}$.
Then for each $n$,
\[
\bigcup_{\lambda \supseteq \mu} \ssyt(\lambda/\mu,a,b,n)
\]
is a strictly column-closed family, where the union is taken over all $\lambda \supseteq \mu$ with at most $l$ rows.
The \defin{row-flagged} Schur polynomials, $s_{\lambda/\mu,a,b}(\xvec)$ in $n$ variables are defined as
\[
 \schur_{\lambda/\mu,a,b}(\xvec) = \sum_{T \in \ssyt(\lambda/\mu,a,b,n)} \xvec^{w(T)},
\]
see \emph{e.g.} \cite{Wachs1985} as a reference.

\subsection{Symplectic fillings}

The following definition is taken from \cite{King1976} and these polynomials are related to representations of $Sp(2n)$.
The \defin{symplectic Schur polynomials}, $\schurSp_\lambda(\xvec)$, in the variables $x_1^{\pm 1},x_2^{\pm 1},\dots,x_n^{\pm 1}$
are defined via fillings of the Young diagram $\lambda$ using the alphabet $1 < \overline{1} < 2 < \overline{2}<\dotsb < n < \overline{n}$
such that rows are weakly increasing, columns are strictly increasing, and entries in row $i$ are greater than or equal to $i$.
Then, for a partition $\lambda$ with at most $n$ parts,
\[
\schurSp_\lambda(\xvec) = \sum_{T \in \spyt(\lambda)} \xvec^{w(T)} \xvec^{-\overline{w}(T)}
\]
where $w(T)$ is the weight only counting unbarred entries and $\overline{w}(T)$ only counts the barred entries.
It is quite clear that the symplectic Young tableaux form a strictly column-closed family,
and that the statistics $w$ and $\overline{w}$ are linear.
Consequently, $\{\schurSp_{k\lambda}\}_{k=1}^\infty$ satisfies a linear recurrence for every fixed partition $\lambda$.

\subsection{Set-valued tableaux and reverse plane partitions}

The \defin{Grothendieck polynomials}\footnote{These are called the \emph{stable} Grothendieck polynomials.} $\grothG_\lambda(\xvec)$ can be defined (see \cite{Buch2002}) as
\[
 \grothG_\lambda(\xvec) = \sum_{T \in \svt(\lambda)} (-1)^{|T|-|\lambda|}\xvec^{w(T)}
\]
where the sum is taken over \defin{set-valued Young tableaux}. These are defined as
fillings of a diagram of shape $\lambda$, but now each box contains a \emph{set} of natural numbers.
For two such sets $A$, $B$ we have $A<B$ if $\max A < \min B$ and similar for $A \leq B$.
With this notation, $\svt(\lambda)$ is the set of all set-valued tableaux (subsets of $[n]$) such that rows are weakly increasing,
and columns are strictly increasing. Here, the $i$th component of $w(T)$ is now the total number of sets where $i$ appears,
and $|T|$ is the sum over all cardinalities of the sets in the boxes.
Note that the lowest-degree part of $\grothG_\lambda(\xvec)$ is the usual Schur polynomial $\schur_\lambda(\xvec)$.

There is also an operator definition of the more
general Grothendieck polynomials which are indexed by permutations and similar to the Schubert polynomials
and introduced by Lascoux and Schützenberger in 1982.

\medskip

To show that $\grothG_\lambda(\xvec)$ satisfy a linear recurrence, one needs to use the more general version of \cref{lem:columnsRecurrence},
since the family of set-valued Young tableaux is not a weakly column-closed family; only columns where each set is a singleton can be duplicated.
However, we note that the family is well-behaved and that every tableau $T \in \svt(k\lambda)$ contains duplicate columns for every $k$
sufficiently large. These observations together with \cref{lem:columnsRecurrence} allows us to deduce that
the Grothendieck polynomials also satisfy a linear recurrence. We leave the details as an exercise to the reader.
This can also be proved using the divided difference operator definition, similar to the Schubert polynomials.

\medskip

Lam and Pylyavskyy \cite{Lam07combinatorialhopf} proved that the \defin{dual stable Grothendieck polynomials}, $\grothg_\lambda(\xvec)$ in $n$ variables
can be defined as
\[
 \grothg_\lambda(\xvec) = \sum_{T \in \rpp(\lambda)} \xvec^{ev(T)}
\]
where $\rpp(\lambda)$ is the set of \defin{reverse plane partitions} of shape $\lambda$,
that is, fillings of $\lambda$ with numbers in $[n]$ such that rows and columns are weakly decreasing.
The statistic $ev(T)_i$ is the total number of columns where $i$ appears.
Evidently, $ev$ is a linear statistic and reverse plane partitions is a well-behaved, strictly column-closed family of fillings.
Consequently, we get a linear recurrence in this case.

\medskip

\subsection{A note on Jack and Macdonald polynomials}

A consequence of satisfying a linear recurrence, is that the sequence of polynomials must satisfy a linear recurrence
under every specialization of the variables. In particular, if we pick a specialization such that all roots of the characteristic polynomial
become equal to $1$, then the resulting sequence is a polynomial.
For example, $k\mapsto \key_{k\alpha}(1^n)$ is a polynomial in $k$.

This observation allows us to deduce that the Jack polynomials $J_\lambda(\xvec,a)$ do \emph{not} satisfy a linear recurrence
for general values of $a$, the sequences $J_{k\lambda}(1^n,a)$ are not of the form given in \cref{eq:generalForm}.
This observation holds for both standard normalizations of Jack polynomials.

It follows that there are no linear recursions for Macdonald polynomials either,
since the Jack polynomials are a specialization of the Macdonald polynomials.

\bibliographystyle{amsalpha}
\bibliography{bibliography}

\end{document}